\documentclass[10pt, reqno]{amsart}
\usepackage{hyperref}
\hypersetup{
	colorlinks=true,
	citecolor=blue,
	linkcolor=blue,
	filecolor=magenta,      
	urlcolor=cyan,
}

\usepackage{geometry}
\usepackage{setspace}
\usepackage[capitalize,noabbrev,nameinlink]{cleveref}
\usepackage{fullpage}
\usepackage{amsmath, amsthm, amssymb, graphicx, dsfont, stmaryrd, extarrows, enumitem}
\usepackage[T1]{fontenc}
\usepackage{textcomp}
\usepackage{lmodern}
\usepackage{microtype} 
\usepackage{amscd}
\usepackage{mathrsfs}
\usepackage{tikz-cd}
\usetikzlibrary{matrix,patterns}
\usepackage[colorinlistoftodos]{todonotes}
\usepackage{color}
\usepackage{bbm}
\usepackage{verbatim}
\usepackage[mathscr]{euscript}
\usepackage{bm}

\numberwithin{equation}{section}
\theoremstyle{plain}
\newtheorem{thm}[equation]{Theorem}
\newtheorem*{thm*}{Theorem}

\newtheorem*{prop*}{Proposition}

\newtheorem{prop}[equation]{Proposition}
    
\newtheorem{cor}[equation]{Corollary}       
\newtheorem{lem}[equation]{Lemma}

\theoremstyle{definition} 
\newtheorem{defn}[equation]{Definition}

\newtheorem{rem}[equation]{Remark}

\newcommand{\Z}{\mathbb{Z}}
\newcommand{\Q}{\mathbb{Q}}

\newcommand{\Hom}{\mathrm{Hom}}
\newcommand{\iHom}{\underline{\mathrm{Hom}}}

\newcommand{\msf}[1]{\mathsf{#1}}

\newcommand{\mc}[1]{\mathcal{#1}}
\newcommand{\mrm}[1]{\mathrm{#1}}

\newcommand{\T}{\mathsf{T}}

\newcommand{\1}{\mathds{1}}

\renewcommand{\mod}[1]{\mathrm{Mod}_{#1}}

\newcommand{\A}{\mathsf{A}}

\newcommand{\D}{\mathsf{D}}

\renewcommand{\mod}[1]{\msf{mod}(#1)}
\newcommand{\Mod}[1]{\msf{Mod}(#1)}
\newcommand{\Flat}[1]{\msf{Flat}(#1)}
\newcommand{\ab}{\msf{Ab}}
\newcommand{\ti}{\textit}

\newcommand{\rlim}{\varinjlim}
\newcommand{\llim}{\varprojlim}
\renewcommand{\t}{\text}
\renewcommand{\c}{\mrm{c}}

\newcommand{\op}{\mrm{op}}

\renewcommand{\Flat}[1]{\msf{Flat}(#1)}
\renewcommand{\flat}[1]{\msf{flat}(#1)}

\newcommand{\closed}[1]{{{\msf{KZg}}_{\msf{Cl}}^{\otimes}}(#1)}

\newcommand{\hspec}[1]{\msf{Spc}^{\msf{h}}(#1)}

\makeatletter
\@namedef{subjclassname@2020}{\textup{2020} Mathematics Subject Classification}
\makeatother
\title{The homological spectrum via definable subcategories}
\begin{document}
\author{Isaac Bird}
\address[Bird]{Department of Algebra, Faculty of Mathematics and Physics, Charles University in Prague, Sokolovsk\'{a} 83, 186 75 Praha, Czech Republic}
\email{bird@karlin.mff.cuni.cz}
\author{Jordan Williamson}
\address[Williamson]{Department of Algebra, Faculty of Mathematics and Physics, Charles University in Prague, Sokolovsk\'{a} 83, 186 75 Praha, Czech Republic}
\email{williamson@karlin.mff.cuni.cz}
\maketitle
\begin{abstract}
We develop an alternative approach to the homological spectrum of a tensor-triangulated category through the lens of definable subcategories. This culminates in a proof that the homological spectrum is homeomorphic to a quotient of the Ziegler spectrum. Along the way, we characterise injective objects in homological residue fields in terms of the definable subcategory corresponding to a given homological prime. We use these results to give a purity perspective on the relationship between the homological and Balmer spectrum.
\end{abstract}
\setcounter{tocdepth}{1}
\tableofcontents

\section{Introduction}
Support theories have long played an important role in pure mathematics: for instance, in commutative algebra, modular representation theory~\cite{BCR, Carlson}, stable homotopy theory~\cite{HS}, and beyond. In the seminal work~\cite{BalSpec}, Balmer proved the existence of a universal support theory, unifying the above examples through \emph{tensor-triangular geometry}, and providing a general framework of support in a plethora of settings. For a rigidly-compactly generated tensor-triangulated category $\T$ with compact objects $\T^\c$, this universal support theory consists of a topological space $\msf{Spc}(\T^\c)$, together with a support function which assigns to every object of $\T^\c$ a subset of $\msf{Spc}(\T^\c)$. The space $\msf{Spc}(\T^\c)$ is called the \emph{Balmer spectrum}, and this universal support theory classifies the thick tensor ideals of $\T^\c$. To appreciate the generality of the framework of tensor-triangular geometry, we refer the reader to~\cite{axiomatic} for a wide range of examples.

A standard but powerful strategy in algebra is descent: passing to residue fields and then reassembling the data to deduce global information using Nakayama's lemma. However, in general tensor-triangulated categories, honest residue fields are not known to exist. As a key aspect of tensor-triangular geometry is the diversity of examples which inhabit it, the nonexistence of residue fields led Balmer-Krause-Stevenson~\cite{bks} to construct \emph{homological residue fields} as replacements, which can be assembled into a topological space $\hspec{\T^\c}$ called the \emph{homological spectrum}~\cite{Balmernilpotence}. However, these homological residue fields are in general hard to compute~\cite{BalmerCameron, CameronStevenson}, but nonetheless have proven their utility through the good behaviour they exhibit, as we now describe.
\newpage

Although the Balmer spectrum is universal, the related notion of support for big (i.e., non-compact) objects does not necessarily behave well~\cite{BF,Zou}. For example, the associated support theory does not satisfy the desired tensor-product formula in general, and moreover, in order for the support to classify localising tensor ideals of $\T$, one must make assumptions regarding the topology on $\msf{Spc}(\T^\c)$~\cite{bhs,Zou}. On the other hand, Balmer~\cite{Balmerhomsupp} showed that the homological spectrum can be used as a basis for a support theory for big objects which always satisfies the tensor-product formula. In addition, Balmer~\cite{Balmernilpotence} proved an abstract nilpotence theorem for tensor-triangulated categories based on the homological spectrum, generalising the nilpotence theorem in stable homotopy theory~\cite{HS} and unifying it with nilpotence theorems in other settings~\cite{Hopkinsglobal, Neemanchromatic,Thomason}. As such, the homological spectrum provides some technical advantages over the Balmer spectrum.

However, there is a significant difference between the homological spectrum and the Balmer spectrum, namely in where the points lie: the Balmer spectrum lives in the compact objects, whereas the homological spectrum lives inside the functors on the compacts. More explicitly, the points of the Balmer spectrum are the prime thick $\otimes$-ideals of $\T^{\c}$, while the points of the homological spectrum are the maximal Serre $\otimes$-ideals of $\mod{\T^{\c}}$, the category of finitely presented right modules over $\T^{\c}$. Consequently, the homological spectrum provides a more appealing entrance to tensor-triangular geometry for those of a functorial disposition.

It is precisely this viewpoint which we develop. More precisely, we formalise the relationship between the homological spectrum and the Ziegler spectrum, thus building a bridge between tensor-triangular geometry and pure homological algebra. In \cref{backgroundsection} we provide the necessary background in purity for those less familiar with these techniques. In particular, since the Ziegler spectrum may be unfamiliar to many readers, let us briefly recall its inception and applications here. 

Originally introduced by Ziegler in the model theory of modules in \cite{Ziegler}, but later expanded to locally coherent categories~\cite{herzog, krspec2}, finitely accessible categories~\cite{kredc, dac}, and most recently triangulated categories~\cite{krsmash}, the Ziegler spectrum of $\T$, denoted $\msf{Zg}(\T)$, is a topological space whose points are given by the indecomposable pure injective objects of $\T$, and whose closed sets biject with Serre subcategories of $\mod{\T^{\c}}$. Understanding big objects in tensor-triangulated categories is a difficult and active area of study~\cite{bigcats}: it is not known that there is a set of localising tensor ideals. The Ziegler spectrum provides one possible solution to this: there is only a set of indecomposable pure injective objects, and these enable one to understand big objects through the definable subcategories they generate. These subcategories are those which can be realised as kernels of finitely presented functors, and have desirable closure properties. Moreover, they are also in bijection with the Serre subcategories of $\mod{\T^{\c}}$. 

Definable subcategories have already shown their utility in the study of triangulated categories: they were used in \cite{krsmash} to show that there is only a set of smashing localisations, for example, and their use is widespread in more representation theoretic applications such as (co)silting and rank functions on triangulated categories (\cite{AMV,discretederived,dualitypairs,conde2022functorial,krcq}, among many others). Consequently, the Ziegler spectrum can provide a striking amount of information about the triangulated category in question.

One can immediately see there is common ground between the Ziegler and homological spectrum, namely the use of Serre subcategories. But this is not enough for us to relate the spaces, since on the Ziegler side these correspond to closed sets, while on the homological they correspond with points. Moreover, the Ziegler spectrum considers \ti{all} Serre subcategories of $\mod{\T^{\c}}$, while the homological spectrum only considers \emph{maximal} Serre $\otimes$-ideals. The body of this paper focusses on how to bypass these obstacles to compare the spectra. 

We first state our main theorem and then put it into context, before giving details of the proof and the other significant results required for it. The culmination of the paper is the recovery of the homological spectrum from a space, $\closed{\T}^\msf{GZ}$, constructed from the Ziegler spectrum.

\begin{thm*}[{\ref{hspecisgz}}]
Let $\T$ be a rigidly-compactly generated tensor-triangulated category. Then there is a homeomorphism \[\Phi\colon \hspec{\T^\c} \to \closed{\T}^\msf{GZ}.\]
\end{thm*}

Motivating the construction of $\closed{\T}^\msf{GZ}$ requires some work, so we defer an explicit definition to later in the introduction. We now proceed to discuss the path to the above theorem.

In \cite{bks}, the foundation is laid on how to construct $\Phi$: associated to a homological prime $\mc{B}$ is a pure injective object $E_{\mc{B}}$, from which one can recover $\mc{B}$. Although in many cases it is known that $E_{\mc{B}}$ is an indecomposable object~\cite{BalmerCameron}, and thus an object of $\msf{Zg}(\T)$, whether $E_{\mc{B}}$ is always indecomposable is not known. Consequently, we cannot apriori construct a map from $\hspec{\T^{\c}}$ to $\msf{Zg}(\T)$. Yet the assignment $\mc{B}\mapsto E_{\mc{B}}$ does provide a faint sketch for how to proceed. The construction of $E_{\mc{B}}$ does not use all the properties of a homological prime - simply the fact that a homological prime is a Serre subcategory. Understanding which additional properties $E_{\mc{B}}$ is endowed with from the fact that $\mc{B}$ is a maximal Serre $\otimes$-ideal is the first step to constructing the map $\Phi$. 

The fact that homological primes are $\otimes$-ideals means that, instead of considering all definable subcategories, we should restrict attention to those which are $\otimes$-closed, and we give more details in \cref{deftensor}. As such, one may retopologise the Ziegler spectrum, with the closed sets in bijection with these $\otimes$-closed definable subcategories; we denote this space $\msf{Zg}^{\otimes}(\T)$, and we let $\msf{Def}^{\otimes}(\msf{X})$ denote the smallest $\otimes$-closed definable subcategory containing a set $\msf{X}$, which corresponds to the closure operation in this new topology. 

A crucial step in the construction of $\Phi$ is understanding $\msf{Def}^{\otimes}(E_{\mc{B}})$, and to do this we investigate the homological residue field corresponding to $\mc{B}$. If $\mc{B}$ is a homological prime, its associated homological residue field is the localisation
\[
\Mod{\T^{\c}}/\rlim\mc{B}.
\]
In general, understanding what the objects in this category look like is extremely difficult. Since the homological residue field is a locally coherent Grothendieck category, it is determined by its injective objects. In \cref{ddc} we investigate the relationship between the injective objects in this localisation and the pure injective objects in the definable category corresponding to $\mc{B}$. This culminates in the following result, which holds true for any Serre subcategory over any compactly generated triangulated category, that is, without any assumption of a closed monoidal structure.

\begin{thm*}[{\ref{determineinjs}}]
Let $\mc{S}$ be a Serre subcategory of $\mod{\T^\c}$, with $Q\colon\Mod{\T^{\c}}\to\Mod{\T^{\c}}/\rlim\mc{S}$ the associated localisation functor, and let $\mathscr{D}(\mc{S})$ be the unique definable subcategory of $\T$ corresponding to $\mc{S}$. Then the composition $Q\circ\bm{y}$ induces an equivalence of categories
\[
Q\circ\bm{y}\colon\mathscr{D}(\mc{S})\cap\msf{Pinj}(\T)\xrightarrow{\sim} \msf{Inj}(\Mod{\T^{\c}}/\rlim\mc{S})
\]
between the injective objects in $\Mod{\T^{\c}}/\rlim\mc{S}$ and the pure injective objects in $\mathscr{D}(\mc{S})$.
\end{thm*}

For context, when the homological residue field arises from a `true' residue field, Balmer-Cameron~\cite{BalmerCameron} and Cameron-Stevenson~\cite{CameronStevenson} have give an explicit form for homological residue fields in terms of comodules over a certain comonad. Our characterisation does not require this assumption, and applies uniformly to all homological residue fields, but only provides information about the injectives.

The following theorem builds on the above result to show how the maximality assumption on a homological prime enables us to characterise the definable $\otimes$-closure of the pure injective $E_{\mc{B}}$.

\begin{thm*}[{\ref{simple}}]
Let $\mc{B}\in\hspec{\T^{\c}}$. Then $\msf{Def}^{\otimes}(E_{\mc{B}})$ contains no proper nonzero $\otimes$-closed definable subcategories. In particular, if $X\in\msf{Def}^{\otimes}(E_{\mc{B}})$ is nonzero, then $\msf{Def}^{\otimes}(X)=\msf{Def}^{\otimes}(E_{\mc{B}})$.
\end{thm*}

Since $\msf{Def}^{\otimes}$ corresponds to the closure operation on $\msf{Zg}^{\otimes}(\T)$, we see that once we pass to the Kolmogorov quotient $\msf{KZg}^\otimes(\T)$, the pure injective $E_{\mc{B}}$ can be replaced by any indecomposable pure injective in its definable closure.

The above theorem enables us to construct a map of sets \[\Phi\colon \hspec{\T^\c} \to \msf{KZg}^\otimes(\T),\] see \cref{map}. However, this map is far from surjective, as $\msf{KZg}^\otimes(\T)$ is, in general, substantially larger than $\hspec{\T^{\c}}$, as we illustrate in the example of $\msf{D}(\Z)$ (cf., \cref{specZ}). It is at this point where we once again use the \ti{maximality} of the points in $\hspec{\T^{\c}}$. This ensures that $\Phi$ factors over
\[
\closed{\T}:=\msf{Cl}(\msf{KZg}^\otimes(\T)),
\]
the closed points of $\msf{KZg}^\otimes(\T)$ equipped with the quotient topology. Using this observation, in \cref{hspecandKQ}, we prove that $\Phi\colon\hspec{\T^\c} \to \closed{\T}$ is a bijection.

However, initial hopes that $\Phi$ could be a homeomorphism are dashed, for, equipped with its usual subspace topology, $\closed{\T}$ is a $T_{1}$-space, while $\hspec{\T^{\c}}$ only has this property in the most restrictive circumstances. Consequently, to obtain a homeomorphism through $\Phi$, we must construct a new topology on $\closed{\T}$.

It is here where we must be careful, for there are two sides to the topological coin when considering the set of indecomposable pure injective objects in $\T$. Firstly there is the Ziegler topology, as stated above, but one can also consider the Zariski topology. What we must keep in mind is that the topology on $\hspec{\T^{\c}}$ is firmly on the Zariski side, while the topologies we have so far been considering have been Ziegler.

In \cref{sec:topologies} we introduce two Zariski-style topologies on $\closed{\T}$ which will be familiar to those who have studied spectra of locally coherent Grothendieck categories. The first topology which we denote by $\closed{\T}^\vee$ is that induced by the $\otimes$-closed Gabriel-Zariski topology on $\Mod{\T^{\c}}$. This is the Hochster dual of the $\otimes$-closed Ziegler topology, thus flipping the aforementioned topological coin. The second topology which we denote by $\closed{\T}^\msf{GZ}$ is also on the Zariski side - it is that induced by the $\otimes$-closed Gabriel-Zariski topology on $\Flat{\T^{\c}}$, see \cref{GZperspective} for further discussion. In particular, both topologies we introduce on $\closed{\T}$ arise perfectly naturally from the perspective of the Gabriel-Zariski topology on Grothendieck categories. It is then an elementary calculation (cf., \cref{closedcts}) to verify that the map $\Phi\colon \hspec{\T^\c} \to \closed{\T}^\msf{GZ}$ is closed and continuous. Combining this with the aforementioned bijection, one obtains that $\Phi$ is a homeomorphism as desired.

We then apply this to give a new perspective on the relationship between the homological spectrum and the Balmer spectrum. Since the Balmer spectrum is universal, there is a continuous map \[\phi\colon \hspec{\T^\c} \to \msf{Spc}(\T^\c)\] for any rigidly-compactly generated tensor-triangulated category $\T$, and moreover this map is always surjective~\cite[Corollary 3.9]{Balmernilpotence}. Even more, $\phi$ is a bijection in all known examples~\cite[\S 5]{Balmernilpotence}, which led Balmer to conjecture that it is always a bijection.  As such, injectivity is the missing piece of the jigsaw, and Barthel-Heard-Sanders~\cite{BHScomparison} proved that this holds if and only if the homological spectrum is $T_0$. We also give an alternative proof of this, see~\cref{T0comparison}. From the point of view of this paper, Balmer's conjecture then can be reduced to understanding the properties of \emph{simple} $\otimes$-closed definable subcategories, see~\cref{seperation} for further discussion.

Lastly, we finish the paper with a detailed example in \cref{sec:example}, where we consider $\D(\Z)$. Firstly, we describe the topological space $\msf{Zg}^{\otimes}(\D(\Z))$. Following this, we go through the above process to identify the topological space $\closed{\D(\Z)}^{\msf{GZ}}$, before giving an explicit homeomorphism between $\closed{\D(\Z)}^{\msf{GZ}}$ and $\msf{Spc}(\D(\Z)^{\c})$, the Balmer spectrum of $\D(\Z)$. This gives a concrete proof that the homological and Balmer spectrum are homeomorphic in this example.

The results of this paper have been utilised in~\cite{definablefunctors} to provide an alternative approach to functoriality of the homological spectrum. This perspective has allowed us to prove functoriality under weaker hypotheses, thereby expanding a result of Balmer~\cite{Balmerhomsupp} beyond geometric functors. 

\subsection*{Acknowledgements}
We are particularly grateful to Mike Prest for his insightful comments and suggestions, as well as for sharing a preliminary version of \cite{PrestWagstaffe} with us. We also thank Henning Krause and Sebastian Opper for their feedback on preliminary versions of this paper, and the anonymous referee for their helpful comments.
Both authors were supported by the project PRIMUS/23/SCI/006 from Charles University, and by the Charles University
Research Center program UNCE/SCI/022.
%\newpage

\section{Purity and definability}\label{backgroundsection}

Throughout this paper we will use the tools of purity and definability in compactly generated triangulated categories and their functor categories. Here we recall the salient definitions and facts which are used throughout, and refer the reader to~\cite{krsmash, krcoh, psl, dac} for more details. 

\subsection{Modules} 
If $\T$ is a compactly generated triangulated category, we let $\T^{\c}$ denote the full subcategory of compact objects, and $\Mod{\T^{\c}}:=\msf{Add}((\T^{\c})^{\op},\ab)$ denote the category of additive functors $(\T^{\c})^{\op}\to\ab$. The \ti{restricted Yoneda embedding}
\[
\bm{y}\colon\T\to\Mod{\T^{\c}}
\]
\[
X\mapsto\Hom_{\T}(-,X)\vert_{\T^{\c}}
\]
is the central tool in relating purity in $\T$ and $\Mod{\T^{\c}}$. The latter is a locally coherent Grothendieck category, whose finitely presented objects are the finitely presented functors, that is, those $f\in\Mod{\T^{\c}}$ which have a presentation of the form
\[
\bm{y}A\to\bm{y}B\to f\to 0
\]
where $A,B\in\T^{\c}$. We write $\mod{\T^\c}$ for the full subcategory of finitely presented functors.
By the usual Yoneda lemma, there is a bijection between $\Hom_{\Mod{\T^{\c}}}(\bm{y}A,\bm{y}B)$ and $\Hom_{\T}(A,B)$ whenever $A$ and $B$ are compact, and thus any $f\in\mod{\T^{\c}}$ is uniquely determined (up to isomorphism) by a map $\alpha\in\Hom_{\T}(A,B)$; more explicitly $f\in\mod{\T^{\c}}$ if and only if $f=\msf{coker}(\bm{y}\alpha)$ for some $\alpha\in\msf{Mor}(\T^{\c})$.

In terms of covariant functors, rather than considering finitely presented functors $\T^{\c}\to\ab$, it is more beneficial to consider the category of \ti{coherent functors}, which we denote by $\msf{Coh}(\T)$. Introduced in \cite{krcoh}, $\msf{Coh}(\T)$ is the full subcategory of $\msf{Add}(\T, \ab)$ consisting of the functors $F\colon\T\to\ab$ which have a presentation of the form
\[
\Hom_{\T}(A,-)\to\Hom_{\T}(B,-)\to F\to 0
\]
where $A,B \in \T^\c$.
Again, it is clear that $F\colon\T\to\ab$ is coherent if and only if $F\simeq\msf{coker}(\Hom_{\T}(\alpha,-))$ for some $\alpha\in\msf{Mor}(\T^{\c})$.
\subsection{Purity and definability} 
Since the restricted Yoneda embedding $\bm{y}$ is cohomological, it sends triangles to long exact sequences. If one picks out the triangles whose image, under $\bm{y}$, is a \ti{short} exact sequence, one obtains the \ti{pure triangles}. More explicitly, a triangle $X\to Y\to Z\to \Sigma X$ is pure if 
\[
0\to \bm{y}X\to\bm{y}Y\to\bm{y}Z\to 0
\]
is exact in $\Mod{\T^{\c}}$. In this case, we call the map $X\to Y$ a pure monomorphism, and say $X$ is a pure subobject of $Y$. We similarly say $Y\to Z$ is a pure epimorphism and $Z$ is a pure quotient of $Y$.

We are now in a position to recall the definition of definable subcategories of triangulated categories.

\begin{defn}
Let $\mc{D}$ be a full subcategory of $\T$. We say that $\mc{D}$ is \emph{definable} if there is a set of coherent functors $\{F_i\}_{I}\subseteq\msf{Coh}(\T)$ such that 
\[
\mc{D} = \{X \in \T : F_i(X) = 0 \text{ for all $i \in I$}\}.
\]
We note that definable subcategories are always closed under pure subobjects and direct products: indeed, this follows from the definition, as coherent functors preserve products, and send pure triangles to short exact sequences by~\cite[Theorem A]{krcoh}.
\end{defn}

Recall from \cite[Theorem 1.8]{krsmash} that an object $X\in\T$ is pure injective provided any of the following equivalent conditions hold:
\begin{enumerate}
\item the object $\bm{y}X$ is an injective object in $\Mod{\T^{\c}}$;
\item the natural map $\Hom_{\T}(M,X)\to\Hom_{\Mod{\T^{\c}}}(\bm{y}M,\bm{y}X)$ is a bijection for all $M\in\T$;
\item any pure monomorphism $X\to Y$ splits.
\end{enumerate}
We write $\msf{Pinj}(\T)$ for the full subcategory of pure injective objects in $\T$. In fact, as shown in \cite[Corollary 1.9]{krsmash}, $\bm{y}$ induces an equivalence of categories between $\msf{Pinj}(\T)$ and $\msf{Inj}(\T^{\c}):=\msf{Inj}(\Mod{\T^{\c}})$, where the latter is the category of injective objects in $\Mod{\T^{\c}}$. There is only a set of indecomposable injective objects in $\Mod{\T^{\c}}$, which is denoted $\msf{inj}(\T^{\c})$, see \cite[Lemma E.1.10]{psl}. Thus there is also only a set of indecomposable pure injective objects in $\T$, which we denote by $\msf{pinj}(\T)$. 

\subsection{The fundamental correspondence} 
As we have seen in the definition, subcategories of coherent functors determine definable subcategories. In this subsection we recall parts of the fundamental correspondence of \cite{krcoh}, which show that this determination is actually unique. 

There is an order-reversing bijection
\begin{equation}\label{SerreDefCoh}
\begin{tikzcd}
\{\text{definable subcategories of $\T$}\} \ar[rr, yshift=1mm, "\mathscr{S}_\mrm{c}(-)"] & & \{\text{Serre subcategories of $\msf{Coh}(\T)$}\} \ar[ll, yshift=-1mm, "\mathscr{D}_\mrm{c}(-)"]
\end{tikzcd}
\end{equation}
where the maps are defined by
\[\mathscr{S}_\mrm{c}(\mc{D}) = \{F \in \msf{Coh}(\T) : FX = 0 \text{ for all $X \in \mc{D}$}\}\] and
\[\mathscr{D}_\mrm{c}(\mc{S}) = \{X \in \T : FX = 0 \text{ for all $F \in \mc{S}$}\}.\]

One can also relate this to $\mod{\T^\c}$ via an antiequivalence introduced in \cite[\S 7]{krcoh} as we now recall. There is a functor $(-)^{\vee}\colon\Mod{\T^{\c}}^\op\to (\T,\ab)$ defined by 
\begin{equation}\label{krdual}
F^{\vee}(X):=\Hom(F,\bm{y}X).
\end{equation}
It is shown in~\cite[Lemma 7.2]{krcoh} that $(-)^{\vee}$ restricts to an exact antiequivalence $\mod{\T^{\c}}\to\msf{Coh}(\T)$. More explicitly, if $f\in\mod{\T^{\c}}$ can be realised as the cokernel of $\bm{y}\alpha$, then $f^{\vee}$ can be realised as the cokernel of $\Hom_\T(\Sigma A,-) \to \Hom_\T(\msf{cone}(\alpha),-)$. Next we give an explicit form for the inverse.

\begin{lem}\label{krdualinverse}
If $G\in\msf{Coh}(\T)$, then 
\[
G^{\vee}(X):=\Hom(G,\msf{h}X)
\]
is a finitely presented functor, where $\msf{h}\colon (\T^\c)^\op \to \msf{Coh}(\T)$ is defined by $\msf{h}X = \Hom_\T(X,-)$. Moreover, $(-)^{\vee}\colon\msf{Coh}(\T)^\op\to\msf{mod}(\T^{\c})$ and $(-)^{\vee}\colon\mod{\T^{\c}}\to\msf{Coh}(\T)^\op$ are mutually inverse.
\end{lem}
\begin{proof}
Suppose that $G \in \msf{Coh}(\T)$ has a presentation 
\[
\Hom_\T(B,-)\to \Hom_\T(A,-)\to G\to 0,
\]
with $A,B\in\T^{\c}$. Applying $\Hom(-,\msf{h}X)$ to this presentation, and using the Yoneda lemma, we see that the sequence
\[
0\to \Hom(G,\msf{h}X)\to\Hom_\T(X,A)\to\Hom_\T(X,B)
\]
is exact. By considering the cocone of the map $A \to B$ representing $G$, we see that $\Hom(G,\msf{h}X)$ is precisely the inverse of the $(-)^{\vee}$ as described in \cite[Lemma 7.2]{krcoh}.
\end{proof}

\begin{rem}\label{orderpreserving}
Note that if $\mc{A}$ is a Serre subcategory of $\mod{\T^{\c}}$, then $\mc{A}^{\vee\vee}=\mc{A}$. As such, if $\mc{A}\subseteq\mc{B}$ is an inclusion of Serre subcategories, then we have $\mc{A}^{\vee}\subseteq\mc{B}^{\vee}$. Indeed, if $F \in \mc{A}^\vee$, then $F^{\vee}\in \mc{A}^{\vee\vee}=\mc{A}\subseteq \mc{B}$. So $F^{\vee\vee}=F\in \mc{B}^{\vee}$. 
\end{rem}

By combining the bijection of \cref{SerreDefCoh} with the exact antiequivalence $(-)^\vee$, we obtain a bijection
\begin{equation}\label{SerreDef}
\begin{tikzcd}
\{\text{definable subcategories of $\T$}\} \ar[rr, yshift=1mm, "\mathscr{S}(-)"] & & \{\text{Serre subcategories of $\mod{\T^\c}$}\} \ar[ll, yshift=-1mm, "\mathscr{D}(-)"]
\end{tikzcd}
\end{equation}
given by
\[\mathscr{S}(\mc{D}) := (\mathscr{S}_\mrm{c}(\mc{D}))^\vee = \{F^\vee : F \in \msf{Coh}(\T) \text{ such that $FX = 0$ for all $X \in \mc{D}$}\}\] and
\[\mathscr{D}(\mc{S}) := \mathscr{D}_\mrm{c}(\mc{S}^\vee) = \{X \in \T : f^\vee X = 0 \text{ for all $f \in \mc{S}$}\}.\]

Since $(-)^\vee$ is order-preserving by \cref{orderpreserving}, and $\mathscr{S}_\mrm{c}(-)$ and $\mathscr{D}_\mrm{c}(-)$ are order-reversing, we see that the bijections $\mathscr{S}(-)$ and $\mathscr{D}(-)$ are order-reversing.

\subsection{The Ziegler spectrum} 
We next recall the definition of the Ziegler spectrum of $\T$. This is the topological space whose points are the (non-zero) indecomposable pure injectives of $\T$, and whose closed sets are of the form $\mc{D}\cap\msf{pinj}(\T)$, where $\mc{D}$ is a definable subcategory of $\T$. We denote this space by $\msf{Zg}(\T)$. 

A key feature of definable subcategories is that they are uniquely determined by the indecomposable pure injective objects contained within them. There is a bijection between definable subcategories $\mc{D}$ of $\T$ and closed subsets of $\msf{Zg}(\T)$, see the statement of the fundamental correspondence in~\cite{krcoh} for more details. From this bijection, one sees that if $\mc{D}$ is a definable subcategory, then for any $X\in\T$ we have
\begin{equation}\label{purelyembedsinproduct}
X\in\mc{D}\iff \t{there is a pure monomorphism }X\to\prod_{I}E_{i}
\end{equation}
where $E_{i}\in\mc{D}\cap\msf{pinj}(\T)$. As a consequence of this, one immediately obtains the following crucial relationship between definable subcategories of $\T$ and pure injective objects: if $\mc{D}\subseteq\T$ is definable, then
\[
\mc{D}=\msf{Def}(\mc{D}\cap\msf{pinj}(\T))
\]
where $\msf{Def}(\msf{X})$ is the smallest definable subcategory containing $\msf{X}$ when $\msf{X}\subseteq\T$ is any class. In particular, for definable $\mc{D}_{1}$ and $\mc{D}_{2}$, we have $\mc{D}_{1}=\mc{D}_{2}$ if and only if $\mc{D}_{1}\cap\msf{pinj}(\T)=\mc{D}_{2}\cap\msf{pinj}(\T)$.

\subsection{Definability along the Yoneda embedding}
So far we have only considered definability in $\T$, but we also need to consider it in $\Mod{\T^{\c}}$, and occasionally, more generally in a finitely accessible category $\A$ with products. A short exact sequence $0\to L\to M\to N\to 0$ in such an $\A$ is \ti{pure exact} if 
\[
0\to\Hom_{\A}(A,L)\to\Hom_{\A}(A,M)\to\Hom_{\A}(A,N)\to 0
\]
is an exact sequence in $\ab$ for any $A\in\msf{fp}(\A)$, where $\msf{fp}(\A)$ denotes the skeletally small subcategory of finitely presented objects in $\A$. We then say $L\to M$ is a pure monomorphism, $L$ is a pure subobject of $M$, $M \to N$ is a pure epimorphism, and $N$ is a pure quotient of $M$. An object $X\in\A$ is pure injective if every pure monomorphism $X\to Y$ splits.

Much of the theory of definability recalled for $\T$ above also holds for a finitely accessible category $\A$ with products. A full subcategory $\mc{D}\subseteq\A$ is definable if and only if there is a set of finitely presented functors $\{f_{i}\}\subseteq\mod{\msf{fp}(\A)^\mrm{op}}$ such that $\mc{D}=\{X\in\A:\bar{f_{i}}X=0 \t{ for all }i\in I\}$, where $\bar{f_{i}}$ is the unique direct limit preserving extension of $f_i$ to $\A$. This is equivalent to $\mc{D}$ being closed under filtered colimits, direct products, and pure subobjects, or equivalently, being closed under direct products, pure subobjects, and pure quotients. We refer the reader to \cite{psl} or \cite{dac} for a comprehensive consideration of definability in finitely accessible categories.

One particular definable subcategory of $\Mod{\T^{\c}}$, which will be frequently used, is that of the flat functors, which we denote by $\Flat{\T^{\c}}$. The following gives a complete description of objects in $\Flat{\T^{\c}}$.

\begin{prop}[{\cite[\S 2.3]{krsmash}}]\label{flatequivalences}
Let $F\in\Mod{\T^{\c}}$. The following are equivalent:
\begin{enumerate}
\item $F\in\Flat{\T^{\c}}$;
\item $F$ is cohomological;
\item $F$ is \ti{fp-injective}, that is $\mrm{Ext}^{1}(g,f)=0$ for all $g\in\mod{\T^{\c}}$;
\item $F$ is an ind-object over $(\T^{\c})^{\op}$, i.e., $F \simeq \rlim_I\bm{y}A_i$ where each $A_i$ is compact and $I$ is filtered.
\end{enumerate}
\end{prop} 
The fact that $\Flat{\T^{\c}}$ is definable in $\Mod{\T^{\c}}$ requires justification. As $\T^{\c}$ is triangulated, it has both weak kernels and weak cokernels, and is also idempotent complete (e.g., see~\cite[Proposition 1.6.8]{Neemantri}). Thus $\Flat{\T^\c}$ is definable~\cite[Theorem 6.1(a)]{dac}. 

We also see from \cref{flatequivalences} that $\Flat{\T^{\c}}$ is a finitely accessible category with products, and we denote its finitely presented objects by $\flat{\T^{\c}}$. We note that $\Flat{\T^\c} \cap \mod{\T^\c} = \flat{\T^\c}$ by~\cite[1.3]{CrawleyBoevey}. It follows that $\bm{y}$ induces a bijection $\T^{\c}\xrightarrow{\sim}\flat{\T^{\c}}$. We may also consider definable subcategories of $\Flat{\T^{\c}}$, which prove to be of vital importance.

We now show that, in relation to understanding purity, there is no difference between looking in $\T$ or in $\Flat{\T^{\c}}$. The following result is almost certainly well known, and is not difficult, but we explicitly include it for referencing purposes. 

\begin{lem}\label{purityflats}
Let $\T$ be a compactly generated triangulated category.   
\begin{enumerate}[label=(\arabic*)]
\item\label{purityflats0} A functor $X \in \Mod{\T^\c}$ is flat and pure injective if and only if it is injective.
\item\label{purityflats1} The restricted Yoneda embedding induces a bijection $\msf{Pinj}(\T)\to\msf{Pinj}(\Flat{\T^{\c}})$.

\item\label{purityflats2} For a definable subcategory $\mc{D}$ of $\T$, we have $\bm{y}(\mc{D} \cap \msf{pinj}(\T)) = \msf{Def}(\bm{y}\mc{D}) \cap \msf{pinj}(\Flat{\T^{\c}})$.

\item\label{purityflats3} There is a bijection between definable subcategories of $\T$ and those of $\msf{Flat}(\T^{\c})$.
\end{enumerate}
\end{lem}

\begin{proof}
For (1), by \cref{flatequivalences}, $X$ is flat if and only if it is fp-injective, which is moreover equivalent to it being absolutely pure by~\cite[Proposition 5.6]{dac}. Consequently if $X$ is flat and pure injective it is absolutely pure and pure injective, and is thus injective by \cite[Lemma 12.3.16]{krbook}. For the converse, any injective is flat by \cref{flatequivalences}, and injectivity implies pure injectivity by definition.

For (2), by \cite[Corollary 1.9]{krsmash}, it suffices to show that an object $X\in\Flat{\T^{\c}}$ is pure injective if and only if it is injective. This is the content of \ref{purityflats0}.

For (3), it is clear that the left hand side is contained in the right, since $\bm{y}$ preserves indecomposability and if $X\in\mc{D}$ then $\bm{y}X\in\msf{Def}(\bm{y}\mc{D})$. For the other inclusion, let $E\in\msf{Def}(\bm{y}\mc{D})\cap\msf{pinj}(\Flat{\T^{\c}})$. Then by \ref{purityflats1}, $E\simeq \bm{y}X$ for some $X\in\msf{pinj}(\T)$. But then $\bm{y}X\in\msf{Def}(\bm{y}\mc{D})$ and hence $X\in\mc{D}$, as $\mc{D} = \{Y \in \T : \bm{y}Y \in \msf{Def}(\bm{y}\mc{D})\}$ by \cite[Corollary 4.4]{AMV}.

Statement (4) follows immediately from \ref{purityflats1} and \ref{purityflats2}. Explicitly, the bijection sends
$
\mc{D}\mapsto \msf{Def}(\bm{y}\mc{D})
$
for $\mc{D}\subseteq\T$ definable, with inverse 
$
\widetilde{\mc{D}}\mapsto \bm{y}^{-1}\widetilde{\mc{D}}=\{X\in\T:\bm{y}X\in\widetilde{\mc{D}}\}
$
for $\widetilde{\mc{D}}\subseteq\Flat{\T^{\c}}$ definable.
\end{proof}

We note that, by \cref{purityflats}, the restricted Yoneda embedding induces a homeomorphism between $\msf{Zg}(\T)$ and $\msf{Zg}(\Flat{\T^{\c}})$, where the latter is defined in an analogous way: the points are the indecomposable pure injective flat functors, and the closed sets are of the form $\mc{D} \cap \msf{pinj}(\Flat{\T^{\c}})$, where $\mc{D}\subseteq\Flat{\T^{\c}}$ is a definable subcategory. 

\begin{rem}
We have only recalled the essential details of purity in triangulated categories required for our needs. The paper~\cite{PrestWagstaffe} provides a much deeper investigation into the model theoretic algebra side of purity in triangulated categories.
\end{rem}

\section{Localisations and definable subcategories}\label{ddc}
In this section, we use the bijection between definable subcategories of $\T$ and Serre subcategories of $\mod{\T^\c}$ to describe the injective objects of the Gabriel quotient $\Mod{\T^{\c}}/\rlim\mc{S}$, where $\mc{S}$ is a Serre subcategory of $\mod{\T^{\c}}$. This will be a crucial tool in better understanding homological residue fields in \cref{sec:homological}. 

Let $\mc{S}\subseteq\T$ be a Serre subcategory of $\mod{\T^\c}$, and recall that we write $\mathscr{D}(\mc{S})$ for the associated definable subcategory of $\T$ (cf., \cref{SerreDef}). 
Given this definable subcategory $\mathscr{D}(\mc{S})$, one may consider the corresponding definable subcategory $\msf{Def}(\bm{y}\mathscr{D}(\mc{S}))\subseteq\Mod{\T^{\c}}$ of \cref{purityflats}\ref{purityflats3}. As such it is natural to investigate how the Serre subcategory $\mc{S}$ of $\mod{\T^\c}$ is related to this definable subcategory of $\Mod{\T^\c}$. The goal of this section is to make this relationship precise. 

We will make use of the following fact, see~\cite[Lemma 2.2.10]{krbook} for instance. 
We write $R$ for the fully faithful right adjoint to the localization functor $Q\colon \Mod{\T^\c} \to \Mod{\T^{\c}}/\rlim\mc{S}$. \begin{lem}\label{quotientisperp}
The functor $R$ restricts to an equivalence of categories $R\colon \Mod{\T^{\c}}/\rlim\mc{S} \xrightarrow{\sim} (\rlim\mc{S})^{\perp}$ with quasi-inverse $Q$, where \[
(\rlim\mc{S})^{\perp}:=\{X\in\Mod{\T^{\c}}:\Hom(\rlim\mc{S},X)=0=\mrm{Ext}^{1}(\rlim\mc{S},X)\}.
\]
\end{lem}

The following result shows that any object in $\msf{Def}(\bm{y}\mathscr{D}(\mc{S}))$ is in $(\rlim\mc{S})^{\perp}$.
\begin{prop}\label{homorthogonal}
Let $f\in\rlim\mc{S}$ and $X\in\msf{Def}(\bm{y}\mathscr{D}(\mc{S}))$. Then $\mrm{Ext}^i(f,X)=0$ for all $i \geq 0$.
\end{prop}
\begin{proof}
We prove the $i=0$ case first. By \cref{purityflats}\ref{purityflats2}, 
$\msf{Def}(\bm{y}\mathscr{D}(\mc{S})) = \msf{Def}(\bm{y}E_j : j \in J)$
where $\{E_{j}\}_J = \mathscr{D}(\mc{S}) \cap \msf{pinj}(\T)$. Consequently if $X\in\msf{Def}(\bm{y}\mathscr{D}(\mc{S}))$, there is a pure monomorphism $0\to X\to \prod_{J}\bm{y}E_{j}$ by the abelian analogue of \cref{purelyembedsinproduct}. By the left exactness of $\Hom(f,-)$ it suffices to show that $\Hom(f,\bm{y}E_{j})=0$ for all $j \in J$. Writing $f=\rlim_{I}f_{i}$, with $f_{i}\in \mc{S}$, we have
\[
\Hom(f,\bm{y}E_{j})\simeq\llim_{I}\Hom(f_{i},\bm{y}E_{j})\simeq \llim_{I}f_{i}^{\vee}(E_{j}).
\]
Since $f_{i}\in\mc{S}$, we have that $f_i^\vee \in \mc{S}^\vee$, but as $E_j \in \mathscr{D}(\mc{S})$, it follows $f_{i}^{\vee}(E_{j})=0$ by definition of $\mathscr{D}(\mc{S})$. This completes the proof of the $i=0$ case, and we now use that $\msf{Def}(\bm{y}\mathscr{D}(\mc{S}))$ is definable, to extend the vanishing property to higher Ext-groups. 

Consider the pure exact sequence $0\to X\to PE(X)\to Z\to 0$, where $PE(X)$ denotes the pure injective hull of $X$. As $\msf{Def}(\bm{y}\mathscr{D}(\mc{S}))$ is definable in $\Flat{\T^\c}$, it follows that each term of this pure exact sequence is in $\msf{Def}(\bm{y}\mathscr{D}(\mc{S}))$ since $X \in \msf{Def}(\bm{y}\mathscr{D}(\mc{S}))$ if and only if $PE(X) \in \msf{Def}(\bm{y}\mathscr{D}(\mc{S}))$ by~\cite[Proposition 12.2.8]{krbook}. Therefore $PE(X)$ is flat and pure injective, and hence is injective by \cite[Lemma 12.3.16]{krbook}, and thus $\t{Ext}^{1}(-,PE(X))=0$. Taking the long exact sequence of $\mrm{Ext}$-groups associated to the pure exact sequence and applying the case $i=0$, one sees that $\t{Ext}^{1}(f,X)=0$. But since $X$ is arbitrary, the same reasoning shows that $\mrm{Ext}^{1}(f,Z)=0$, hence by dimension shifting we see that all higher extensions also vanish, which proves the claim.
\end{proof}

As mentioned, the previous result gives the following.
\begin{cor}\label{definlocalisation}
There is an inclusion $\msf{Def}(\bm{y}\mathscr{D}(\mc{S}))\subseteq(\rlim \mc{S})^{\perp}$. \qed
\end{cor}

Since $\mc{S}$ is a Serre subcategory of $\mod{\T^{\c}}$, the class $\rlim\mc{S}$ is a hereditary torsion class of $\Mod{\T^{\c}}$ (see, for example \cite[Proposition 11.1.36]{psl}). Let $(\rlim\mc{S},\mc{F})$ denote the corresponding hereditary torsion pair, which, by being hereditary, is uniquely determined by a class of injective objects in $\Mod{\T^{\c}}$; in other words, there is a class $\mc{E}\subseteq\msf{Inj}(\T^{\c})$ such that 
\begin{equation}\label{limclosure}
\rlim\mc{S}=\{F\in\Mod{\T^{\c}}:\Hom(F,E)=0 \t{ for all }E\in\mc{E}\}
\end{equation}
and 
\begin{equation}
\mc{F}=\{F\in\Mod{\T^{\c}} : \t{there is an embedding } F\to X \t{ with }X\in\msf{Prod}(\mc{E})\},
\end{equation}
where $\msf{Prod}(\mc{E})$ denotes the subcategory of retracts of direct products of elements of $\mc{E}$. Thus $\mrm{Inj}(\T^\c) \cap \mc{F} = \mc{E}$.

By \cref{definlocalisation}, we see that if $X \in \mathscr{D}(\mc{S})$ is pure injective, then $\bm{y}X \in \mc{E}$. In particular, $\msf{Def}(\bm{y}\mathscr{D}(\mc{S})) \cap \msf{Inj}(\T^\c) \subseteq \mc{E}$. However, it could be the case that there are many more injective objects in $\mc{E}$ than in $\msf{Def}(\bm{y}\mathscr{D}(\mc{S}))$. The following result shows that this is in fact not the case.

\begin{lem}\label{injcogen}
Let $\mc{E}$ be the set of injectives that cogenerates the torsion pair $(\rlim\mc{S},\mc{F})$. Then $\mc{E}\subseteq\bm{y}\mathscr{D}(\mc{S})$. In particular, if $E \in (\rlim\mc{S})^\perp \cap \msf{Inj}(\T^\c)$, then $E \simeq \bm{y}X$ for some $X \in \mathscr{D}(\mc{S}) \cap \msf{Pinj}(\T)$.
\end{lem}
\begin{proof}
If $E\in\mc{E}$ then $E\simeq\bm{y}X$ for some $X\in\msf{Pinj}(\T)$. If $f\in\mc{S}$, then by \cref{limclosure}, we have
\[
0=\Hom(f,E)\simeq \Hom(f,\bm{y}X)= f^{\vee}(X).
\]
Consequently $X\in\mathscr{D}(\mc{S})$ by definition of $\mathscr{D}(\mc{S})$, and thus $E \simeq \bm{y}X\in\bm{y}\mathscr{D}(\mc{S})$.
\end{proof}

We can now assemble these results to give the main result of this section.
\begin{thm}\label{determineinjs}
The functor $Q\circ \bm{y}\colon\T\to\Mod{\T^{\c}}/\rlim\mc{S}$ restricts to an equivalence of categories
\[
\mathscr{D}(\mc{S})\cap\msf{Pinj}(\T)\xrightarrow{\sim}\msf{Inj}(\Mod{\T^{\c}}/\rlim\mc{S}).
\]
\end{thm}
\begin{proof}
The functor $Q\colon (\rlim\mc{S})^{\perp} \to \Mod{\T^\c}/\rlim\mc{S}$ is an exact equivalence by \cref{quotientisperp}, and therefore it restricts to an equivalence \[Q\colon \msf{Inj}((\rlim\mc{S})^{\perp}) \to \msf{Inj}(\Mod{\T^\c}/\rlim\mc{S})\] on injective objects. As $\msf{Inj}((\rlim\mc{S})^{\perp}) = (\rlim \mc{S})^\perp \cap \msf{Inj}(\T^\c)$ (for example, see~\cite[Corollary 2.2.15]{krbook}), it thus remains to prove that
\[\bm{y}\colon \mathscr{D}(\mc{S}) \cap \msf{Pinj}(\T) \xrightarrow{\sim} (\rlim \mc{S})^\perp \cap \msf{Inj}(\T^\c)\] is an equivalence of categories. This is well-defined by \cref{definlocalisation}, essentially surjective by \cref{injcogen}, and fully faithful by~\cite[Corollary 1.9]{krsmash}.
\end{proof}

\section{An atlas of spectra}\label{sec:homological}
In this section, we prove the main results of this paper. Following some preliminaries on monoidal considerations, we give a characterization of the injective objects in homological residue fields in terms of definability, and then prove that the homological spectrum is homeomorphic to a certain space built from a quotient of the Ziegler spectrum using a Gabriel-Zariski style topology.

Henceforth we assume that $\T$ is a rigidly-compactly generated tensor-triangulated category. We briefly recall what this means below, together with the interaction between the tensor product on $\T$ and on its module category, see~\cite{bks} for more details.

\subsection{Tensors and modules}
Suppose that $\T$ is a rigidly-compactly generated tensor-triangulated category. More precisely, this is a triangulated category equipped with a closed symmetric monoidal structure, with both the tensor product $- \otimes -$ and the internal hom $\iHom(-,-)$ exact in both variables. We write $\1$ for the monoidal unit. Moreover the rigid and compact objects of $\T$ coincide, where an object $X \in \T$ is said to be rigid if the natural map $\iHom(X,\1) \otimes Y \to \iHom(X,Y)$ is an isomorphism for all $Y \in \T$. 

When $\T$ is a rigidly-compactly generated tensor-triangulated category, the module category $\Mod{\T^\c}$ inherits a symmetric monoidal structure via Day convolution, which is characterized by the properties that $- \otimes -$ commutes with colimits in both variables, and that $\bm{y}(X \otimes Y) = \bm{y}X \otimes \bm{y}Y$ for all $X,Y \in \T^\c$. Then the restricted Yoneda embedding $\bm{y}\colon \T \to \Mod{\T^\c}$ is symmetric monoidal. Moreover, the module category $\Mod{\T^\c}$ is \emph{closed} symmetric monoidal, so that it admits an internal hom.

For any $X \in \T$ (not necessarily compact), $\bm{y}X$ is $\otimes$-flat (i.e., $\bm{y}X \otimes -$ is exact)~\cite[Proposition A.14]{BKSframe}. From this one sees that the flat objects in $\Mod{\T^\c}$ coincide with the $\otimes$-flat objects. Indeed, given any $M \in \Flat{\T^\c}$, we have $M \simeq \rlim \bm{y}X_i$ where $X_i \in \T^\c$, and this is $\otimes$-flat by the above recollection. Conversely, suppose that $M$ is $\otimes$-flat, and $X \to Y \to Z$ is a triangle of compacts. Writing $D = \iHom(-, \1)$ for the functional dual, we see that \[\bm{y}(D\Sigma X) \to \bm{y}(DZ) \to \bm{y}(DY) \to \bm{y}(DX)\] is exact, and then as $M(X \otimes -) \simeq M \otimes \bm{y}(DX)$ by~\cite[Lemma 2.5]{bks} it follows that \[M(\Sigma X) \to M(Z) \to M(Y) \to M(X)\] is exact. Therefore $M$ is cohomological, and hence flat by \cref{flatequivalences}.

\subsection{Definable tensor closure}\label{deftensor}
In this subsection we give some preparatory results regarding the interaction between tensors and definable subcategories.

We say a definable subcategory $\mc{D}$ is \emph{$\otimes$-closed} if whenever $X \in \mc{D}$, we have $Y \otimes X \in \mc{D}$ for all $Y \in \T$. We emphasise that such definable subcategories are not assumed to be triangulated. By~\cite[Theorem 5.1.8]{wagstaffe}, the fundamental correspondence (cf., \cref{SerreDef}) restricts to an order-reversing bijection between $\otimes$-closed definable subcategories of $\T$ and Serre $\otimes$-ideals of $\mod{\T^\c}$, that is, those Serre subcategories $\mc{S}$ for which $f\otimes g\in\mc{S}$ whenever $g\in\mc{S}$ and $f\in\mod{\T^{\c}}$

The following lemma shows that the $\otimes$-closure of definable subcategories can be tested on compact objects. We will give a direct proof in the spirit of this paper using functor categories. We note that an alternative proof can be found in~\cite[Lemma 5.1.11]{wagstaffe}.
\begin{lem}\label{deftensoridealcompacts}
Let $\mc{D}\subseteq\T$ be a definable subcategory. Then $\mc{D}$ is $\otimes$-closed if and only if $C\otimes X\in\mc{D}$ for every $X\in\mc{D}$ and $C\in\T^{\c}$.
\end{lem}
\begin{proof}
Only the reverse implication needs proving. Suppose $D \in \mc{D}$ and $Z\in\T$. Consider $\bm{y}Z\in\Flat{\T^{\c}}$ which we will realise as $\bm{y}Z\simeq \rlim\bm{y}C_{i}$, where each $C_{i}\in\T^\c$. Now, since $\mc{D}=\{X\in\T: \bm{y}X\in\msf{Def}(\bm{y}\mc{D})\}$ by~\cite[Corollary 4.4]{AMV}, to prove that $Z \otimes D \in \mc{D}$ it suffices to show that $\bm{y}(Z\otimes D)\in\msf{Def}(\bm{y}\mc{D})$. This holds, as there are isomorphisms
\[
\bm{y}(Z\otimes D)\simeq \bm{y}Z\otimes\bm{y}D\simeq \rlim(\bm{y}C_i\otimes\bm{y}D)\simeq \rlim\bm{y}(C_i\otimes D),
\]
and by assumption we have $C_i\otimes D\in\mc{D}$ and $\msf{Def}(\bm{y}\mc{D})$ is closed under direct limits. 
\end{proof}

Using that $\otimes$-closure of definable subcategories may be tested on compacts, we obtain the following characterization of definable $\otimes$-closed subcategories.
\begin{lem}\label{definabletensorcharacterisation}
Let $\mc{D}$ be a full subcategory of $\T$. Then $\mc{D}$ is a $\otimes$-closed definable subcategory if and only if there exists a set $\{F_i\}_I$ of coherent functors such that \[\mc{D} = \{X \in \T : F_i(C \otimes X) = 0 \text{ for all $C \in \T^\c$ and $i \in I$}\}.\]
\end{lem}
\begin{proof}
One sees that for any set $\{F_i\}_I$ of coherent functors, $\{X \in \T : F_i(C \otimes X) = 0 \text{ for all $C \in \T^\c$ and $i \in I$}\} 
$
is definable and $\otimes$-closed using \cref{deftensoridealcompacts}, which proves the reverse implication. For the forward implication, let $\mc{D}$ be a $\otimes$-closed definable subcategory. Since $\mc{D}$ is definable there is a set $\{F_i\}_I$ of coherent functors so that 
$\mc{D} = \{X \in \T : F_i(X) = 0 \text{ for all $i \in I$}\}.$ Since $\mc{D}$ is $\otimes$-closed, we have $X \in \mc{D}$ if and only if $C \otimes X \in \mc{D}$ for all $C \in \T^\c$. In other words, $F_iX = 0$ for all $i \in I$ if and only if $F_i(C \otimes X) = 0$ for all $i \in I$ and $C \in \T^\c$. As such, $\mc{D} = \{X \in \T : F_i(C \otimes X) = 0 \text{ for all $C \in \T^\c$ and $i \in I$}\}$ as required.
\end{proof}

With the addition of a tensor product it makes sense to retopologise the Ziegler spectrum. Rather than considering closed sets bijecting with definable subcategories, we follow \cite{wagstaffe} and consider the Ziegler $\otimes$-topology, whose closed sets are of the form $\mc{D} \cap \msf{pinj}(\T)$ where $\mc{D}$ is a $\otimes$-closed definable subcategory. We denote this topological space by $\msf{Zg}^\otimes(\T)$.

The following result enables one to relate the definable closure of a set of objects in the Ziegler $\otimes$-topology with its closure in the usual Ziegler topology.

\begin{lem}\label{relatingclosures}
Let $\msf{X}\subseteq\T$. Then
\[
\msf{Def}^{\otimes}(\msf{X})=\msf{Def}(C\otimes X:C\in\T^{\c},X\in\msf{X}).
\]
\end{lem}

\begin{proof}
Since $\bm{y}$ is strong monoidal, we may use \cref{purityflats} to see that it suffices to show
$
\mc{D}:=\msf{Def}(\bm{y}C\otimes\bm{y}X\colon C\in\T^{\c},X\in\msf{X})
$
is closed under tensoring with $\bm{y}A$ for any $A\in\T^{\c}$. By the comment preceding \cite[Proposition 3.4.9]{psl}, an object $F\in\Flat{\T^{\c}}$ lies in $\mc{D}$ if and only if there is a pure monomorphism
$F \to U$
where $U$ is an object in the closure of $\{\bm{y}C\otimes\bm{y}X\colon C\in\T^{\c},X\in\msf{X}\}$ under direct limits and products. 

So, suppose $Z\in\mc{D}$ with $Z\to U$ the corresponding pure monomorphism and let $A$ be a compact object. Since $\bm{y}A$ is a finitely presented functor, the map
$
\bm{y}A\otimes Z\to\bm{y}A\otimes U
$
is also a pure monomorphism, by \cite[\S13]{dac}. Since $U$ is built from direct limits and products of objects in $\{\bm{y}C\otimes\bm{y}X\colon C\in\T^{\c},X\in\msf{X}\}$, and tensoring with $\bm{y}A$ commutes with taking direct limits and products, it follows that $\bm{y}A\otimes U$ is also in the closure of $\{\bm{y}C\otimes\bm{y}X\colon C\in\T^{\c},X\in\msf{X}\}$ under direct limits and direct products, and thus $\bm{y}A\otimes Z$ is an object in $\mc{D}$, as we wanted.
\end{proof}

\subsection{Homological residue fields}
In this section we recall the construction of the homological spectrum from~\cite{Balmernilpotence, bks}, and use the results of \cref{ddc} to describe the injectives in homological residue fields. 

The \emph{homological spectrum} of $\T$, which is denoted $\hspec{\T^{\c}}$, is the topological space whose points are maximal proper Serre $\otimes$-ideals of $\mod{\T^{\c}}$. The points of the homological spectrum are called \emph{homological primes}. We will describe and discuss the topology in due course. 

Associated to any Serre $\otimes$-ideal $\mc{S}$ of $\mod{\T^{\c}}$ is a pure injective object of $\T$, denoted $E_{\mc{S}}$. This is obtained as follows~\cite[\S 3]{bks}. There is a localisation sequence
\[
\begin{tikzcd} 
\T \arrow[r, "\bm{y}"] & \Mod{\T^{\c}} \arrow[r, shift left = 1mm, "Q"] \arrow[r, hookleftarrow, shift left = -1mm, "R"']
& \Mod{\T^{\c}}/\rlim\mc{S} \end{tikzcd}
\]
where $Q$ is the localisation at the hereditary torsion class $\rlim\mc{S}$ and $R$ is the fully faithful right adjoint to $Q$. Consider the object $Q\bm{y}\1\in\Mod{\T^{\c}}/\rlim\mc{S}$, where $\1$ is the $\otimes$-unit in $\T$. As this quotient category is Grothendieck, there is an injective hull $E(Q\bm{y}\1)$ of $Q\bm{y}\1$. The functor $R$ preserves injectives, hence $RE(Q\bm{y}\1)\simeq\bm{y}E_{\mc{S}}$ for a unique pure injective object $E_\mc{S}$ in $\T$. 

The topology on $\hspec{\T^\c}$ is then defined to have a basis of closed sets given by
\[\mrm{supp}^\msf{h}(x) := \{\mc{B} \in \hspec{\T^\c} : \iHom(x, E_\mc{B}) \neq 0\}\]
where $x$ ranges over the compacts $\T^\c$. We note that this not the definition of $\mrm{supp}^\msf{h}$ given in~\cite{Balmernilpotence}, but they are the same by~\cite[Proposition 4.4]{Balmerhomsupp}. For $\mc{B} \in \hspec{\T^\c}$, the Grothendieck category $\Mod{\T^{\c}}/\rlim\mc{B}$ is called the \emph{homological residue field} (associated to $\mc{B}$). 

Given a homological prime $\mc{B}$ there are two ways to obtain a $\otimes$-closed definable subcategory: firstly, via $\mathscr{D}(\mc{B})$ from the fundamental correspondence (cf., \cref{SerreDef}) or via the assignment $\msf{Def}^\otimes(E_\mc{B})$. The following result shows that these two approaches agree. By a \ti{simple} $\otimes$-closed definable subcategory, we mean one which contains no nonzero proper $\otimes$-closed definable subcategories.

\begin{prop}\label{simple}
Let $\mc{B}\in\hspec{\T^{\c}}$. Then $\msf{Def}^{\otimes}(E_{\mc{B}}) = \mathscr{D}(\mc{B}).$ In particular, $\msf{Def}^{\otimes}(E_{\mc{B}})$ is a simple $\otimes$-closed definable subcategory of $\T$.
\end{prop}

\begin{proof}
Since $Q\bm{y}E_{\mc{B}}$ is nonzero in $\Mod{\T^{\c}}/\rlim\mc{B}$, it follows that $RQ\bm{y}E_\mc{B} \in (\rlim\mc{B})^\perp$ by \cref{quotientisperp}. By construction of $E_\mc{B}$, we have \[RQ\bm{y}E_\mc{B} = RQRE(Q\bm{y}\1) = RE(Q\bm{y}\1) = \bm{y}E_\mc{B}.\] By \cref{injcogen} we see that $E_\mc{B} \in \mathscr{D}(\mc{B})$, and therefore $\msf{Def}^{\otimes}(E_{\mc{B}})\subseteq\mathscr{D}(\mc{B})$. Since the bijection of \cref{SerreDef} is order-reversing and $\mc{B}$ is a maximal Serre $\otimes$-ideal, we have that $\mathscr{D}(\mc{B})$ is a simple $\otimes$-closed definable subcategory. Thus $\mathscr{D}(\mc{B})=\msf{Def}^{\otimes}(E_{\mc{B}})$ as required.
\end{proof}

By combining the previous result with \cref{determineinjs}, we obtain the following characterisation of the injectives in homological residue fields. 

\begin{cor}
Let $\mc{B}\in\hspec{\T^{\c}}$. Then an object $E\in\Mod{\T^{\c}}/\rlim\mc{B}$ is injective if and only if it is of the form $Q\bm{y}X$ with $X\in\msf{Def}^{\otimes}(E_{\mc{B}})\cap\msf{Pinj}(\T)$. Moreover, there is an equivalence of categories \[Q\circ\bm{y}\colon \msf{Def}^\otimes(E_\mc{B}) \cap \msf{Pinj}(\T) \xrightarrow{\sim} \msf{Inj}(\Mod{\T^\c}/\rlim\mc{B}).\] 
\end{cor}

\begin{rem}
We note that the preceding corollary is a refinement of \cite[Corollary 2.18(c)]{bks}, in which it is shown that every injective object in $\Mod{\T^{\c}}/\rlim\mc{B}$ is of the form $Q\bm{y}E$ for some unique pure injective $E\in\T$. We have restricted the domain to show that $E$ must lie in the small definable subcategory $\msf{Def}^{\otimes}(E_{\mc{B}})$.
\end{rem}

\subsection{Comparing the Ziegler and homological spectrum}\label{sec:topologies}
In this section, we prove our main result, \cref{hspecisgz}, giving a new perspective on the homological spectrum via the Ziegler spectrum. 

Since there are several spectra and topologies involved, we give a schematic relating them in \cref{underlyingsets}, which also shows our line of approach. All the maps below are maps of the underlying sets of topological spaces. We provide a figure showing the topological properties of maps below in \cref{maps}.

\begin{figure}[htb]
\centering
\begin{tikzcd}[column sep = 1.8cm, cells={nodes={draw=black}}]
\msf{Zg}(\T) & \msf{Zg}^{\otimes}(\T) \arrow[l, shorten=2mm, "\mrm{id}" {yshift=3pt}, swap, "\cong" description] \arrow[r, shorten=2mm, "\mrm{quotient}" {yshift=3pt}, twoheadrightarrow] & \msf{KZg}^\otimes(\T) & \closed{\T} \arrow[l, shorten=2mm, "\mrm{inc}" {yshift=3pt}, swap, rightarrowtail] & \hspec{\T^{\c}} \arrow[l, shorten=2mm,"\Phi" {yshift=3pt}, "\t{Prop. }\ref{hspecandKQ}" ' {yshift = -3pt} , swap, "\cong" description]
\end{tikzcd}
\caption{A schematic of the underlying sets of the topological spaces under consideration.}
\label{underlyingsets}
\end{figure}

Firstly, we consider the Kolmogorov quotient of $\msf{Zg}^\otimes(\T)$, which we denote by $\msf{KZg}^\otimes(\T)$. Recall that this is the quotient space $\msf{Zg}^\otimes(\T)/\!\!\sim$ where \[X \sim Y \quad  \text{if and only if} \quad \msf{Def}^\otimes(X) \cap \msf{pinj}(\T) = \msf{Def}^\otimes(Y) \cap \msf{pinj}(\T).\] The Kolmogorov quotient is the best approximation to $\msf{Zg}^\otimes(\T)$ by a $T_0$-space; more precisely, the operation of taking the Kolmogorov quotient is left adjoint to the inclusion of the $T_0$-spaces into all topological spaces, and exhibits the $T_0$-spaces as a reflective subcategory of all spaces.

Let $\mc{B}$ be a homological prime, and consider the associated pure injective object $E_{\mc{B}}$. For any $E \in \msf{Def}^\otimes(E_\mc{B}) \cap \msf{pinj}(\T)$, we have 
\begin{equation}\label{defequal}
\msf{Def}^{\otimes}(E) = \msf{Def}^{\otimes}(E_{\mc{B}})
\end{equation}
by \cref{simple}. Consequently all elements of $\msf{Def}^\otimes(E_\mc{B}) \cap \msf{pinj}(\T)$ become equal in $\msf{KZg}^\otimes(\T)$. As such we have proved the following.

\begin{lem}\label{map}
There is a well-defined map \[\Phi\colon\hspec{\T^{\c}}\to \msf{KZg}^\otimes(\T)\] defined by $\Phi(\mc{B}) = [E]$ where $E$ is any indecomposable pure injective in $\msf{Def}^\otimes(E_\mc{B})$. \qed
\end{lem}

\begin{rem}\label{notT0}
We note that $\msf{Zg}^\otimes(\T)$ is in general not a $T_0$-space, so that taking the Kolmogorov quotient is necessary in the previous result. Indeed, if the shift $\Sigma$ on $\T$ is not the identity, then for any $E \in \msf{Zg}^\otimes(\T)$ we also have $\Sigma E \in \msf{Zg}^\otimes(\T)$, and these points are not topologically distinguishable since $\msf{Def}^\otimes(E) = \msf{Def}^\otimes(\Sigma E)$.
\end{rem}

\begin{rem}
In the case when each $E_{\mc{B}}$ is itself indecomposable, the map $\Phi$ can be chosen to send $\mc{B}$ to $E_{\mc{B}}$. However, in general it is not known that $E_{\mc{B}}$ is indecomposable, although this holds in many cases, see~\cite{BalmerCameron} for some examples.
\end{rem}

\begin{defn}
Write $\closed{\T}$ for the set of closed points of $\msf{KZg}^\otimes(\T)$. 
\end{defn}

The following result provides the first key input to our proof of \cref{hspecisgz}.
\begin{prop}\label{hspecandKQ}
The map of sets $\Phi\colon\hspec{\T^{\c}}\to \msf{KZg}^\otimes(\T)$ of \cref{map} induces a bijection \[\Phi\colon \hspec{\T^{\c}} \xrightarrow{\cong} \closed{\T}.\]
\end{prop}

\begin{proof}
Let us first show that the image of the map lies in $\closed{\T}$. Let $\mc{B}$ be a homological prime and let $E$ be a representative for $\Phi(\mc{B})$. To show that $\Phi(\mc{B})$ is closed in $\msf{KZg}^\otimes(\T)$ we must show that if $P \in \msf{Def}^\otimes(E) \cap \msf{pinj}(\T)$, then $P \sim E$. 
So suppose that $P\in\msf{Def}^{\otimes}(E) \cap \msf{pinj}(\T)$. By \cref{simple} together with \cref{defequal}, we know that $\msf{Def}^{\otimes}(E)$ contains no non-trivial proper $\otimes$-closed definable subcategories. Consequently $\msf{Def}^{\otimes}(E)=\msf{Def}^{\otimes}(P)$. Thus $E\sim P$ as required. Therefore we have a well-defined map $\Phi\colon \hspec{\T^\c} \to \closed{\T}$. 

For injectivity, suppose that $\Phi(\mc{B}) = \Phi(\mc{B}')$. It then follows from \cref{defequal} that $\msf{Def}^\otimes(E_\mc{B}) = \msf{Def}^\otimes(E_{\mc{B}'})$. Let $f \in \mod{\T^\c}$, and consider the set \[\mathcal{A}_f := \{Y \in \T : \bm{y}Y \otimes f \simeq 0\}.\] As $f$ is finitely presented, $\mathcal{A}_f$ is a $\otimes$-closed definable subcategory of $\T$. Indeed, we may consider the functor $-\otimes f:\Mod{\T^{\c}}\to\Mod{\T^{\c}}$. This is a definable functor in the sense of \cite[\S 13]{dac}, and therefore its kernel $\mc{K} := \msf{ker}(- \otimes f)$ is definable by \cite[Proposition 13.3]{dac}. Clearly $\mc{K}$ is also $\otimes$-closed in $\Mod{\T^{\c}}$. By \cref{purityflats}\ref{purityflats3}, we see that $\mc{A}_{f}$ is the unique definable subcategory of $\T$ corresponding to $\mc{K} \cap \Flat{\T^\c}$, and as $\bm{y}$ is strong monoidal, it follows that $\mc{A}_{f}$ is a $\otimes$-closed definable subcategory, as claimed.

Now, if $\bm{y}E_{\mc{B}}\otimes f\simeq 0$, then $E_{\mc{B}}\in\mc{A}_{f}$. As the latter is a $\otimes$-closed definable subcategory of $\T$, it follows that $\msf{Def}^{\otimes}(E_{\mc{B}})\subseteq \mc{A}_{f}$, but immediately this yields that $\msf{Def}^{\otimes}(E_{\mc{B}'})\subseteq\mc{A}_{f}$; hence $E_{\mc{B}'}\in\mc{A}_{f}$ showing that $\bm{y}E_{\mc{B}'}\otimes f\simeq 0$. Consequently we see that $\bm{y}E_{\mc{B}}\otimes f\simeq 0$ if and only if $\bm{y}E_{\mc{B}'}\otimes f\simeq 0$ for any $f\in\mod{\T^{\c}}$. But by~\cite[Theorem 3.5]{bks}, we know that $\mc{B} = \msf{ker}(\bm{y}E_\mc{B} \otimes -) \cap \mod{\T^\c}$ and similarly for $\mc{B}'$. Hence, this shows that $\mc{B} = \mc{B}'$ as desired.

Lastly, we show that the map is surjective. Let $X$ be a closed point in $\msf{KZg}^\otimes(\T)$ and consider the $\otimes$-closed definable subcategory $\msf{Def}^{\otimes}(X)$. Let $\mc{B}=\mathscr{S}(\msf{Def}^{\otimes}(X))$ denote the associated Serre $\otimes$-ideal of $\mod{\T^\c}$. If $\mc{B}\subseteq\mc{A}$, with $\mc{A}$ a Serre $\otimes$-ideal, then there is a reverse inclusion $\mathscr{D}(\mc{A})\subseteq\msf{Def}^{\otimes}(X)=\mathscr{D}(\mc{B})$ on the corresponding $\otimes$-closed definable subcategories. As $X$ is a closed point, we therefore have $\mathscr{D}(\mc{A}) =\msf{Def}^\otimes(X)$ and hence $\mc{B} = \mc{A}$. Therefore $\mc{B}$ is a maximal Serre $\otimes$-ideal of $\mod{\T^\c}$, and hence $\mc{B} \in \hspec{\T^{\c}}$. Thus we have obtained a candidate homological prime, whose image under $\Phi$ we now show is $X$. We have $\msf{Def}^\otimes(X) = \mathscr{D}(\mc{B}) = \msf{Def}^\otimes(E_\mc{B})$ by \cref{simple}, and hence for any $E \in \msf{Def}^\otimes(E_\mc{B})$ we see that $\msf{Def}^\otimes(E) = \msf{Def}^\otimes(X)$ by \cref{defequal}, so $\Phi(\mc{B}) = X$ in $\msf{KZg}^\otimes(\T)$ as required.
\end{proof}

\begin{rem}
From the proof of the previous result, one sees that the inverse map $\Phi^{-1}\colon\closed{\T}\to \hspec{\T^{\c}}$ is given by
\[
X\mapsto \mathscr{S}(\msf{Def}^{\otimes}(X)).
\]
\end{rem}

\begin{rem}\label{subspaceT1}
One may hope that the map given above is a homeomorphism when $\closed{\T}$ is equipped with the subspace topology, but this cannot be the case in general. Note that by definition $\closed{\T}$ would then be a $T_{1}$-space; that is every point in it is closed. If $\hspec{\T^{\c}}$ were homeomorphic to it, it would follow that $\hspec{\T^{\c}}$ would then also be $T_{1}$ and hence $T_0$. The $T_{0}$ property would ensure that the Balmer spectrum $\msf{Spc}(\T^{\c})$ were homeomorphic to $\hspec{\T^{\c}}$~\cite[Proposition 4.5]{BHScomparison}. For any commutative ring $R$, there is an homeomorphism $\msf{Spec}(R)\simeq \msf{Spc}(\D(R)^{\c})$. Were the latter $T_{1}$, then the former would also be, but $\msf{Spec}(R)$ is known to be $T_{1}$ if and only if $R$ has Krull dimension zero. Consequently, as claimed, there cannot be a homeomorphism between $\closed{\T}$ and $\hspec{\T^{\c}}$ in general when the former is equipped with the subspace topology. Below we show that, with $\closed{\T}$ retopologised with the $\msf{GZ}$-topology, $\Phi$ does induce a homeomorphism $\hspec{\T^{\c}}\to\closed{\T}^\msf{GZ}$, see \cref{hspecisgz}.
\end{rem}

In light of the previous remark, we want to consider an alternative topology on the Ziegler spectrum. As such, we state the following for motivation. Since the result is only for motivation we omit the proof and note that it follows as in the non-tensor case~\cite[Corollary 4.5]{krspec2}.
\begin{lem}\label{compactopens}
The compact opens in $\msf{Zg}^\otimes(\T)$ are precisely the sets of the form \[(F)_\otimes = \{X \in \msf{pinj}(\T) : F(C \otimes X) \neq 0 \text{ for some $C \in \T^\c$}\}\] where $F \in \msf{Coh}(\T)$.
\end{lem}

Although $\msf{Zg}^\otimes(\T)$ is not a spectral space in general (see \cref{notT0}), we may mimic the construction of the Hochster dual, to define a space $\msf{Zg}^\otimes(\T)^\vee$. This assignment is just no longer a duality. In the space $\msf{Zg}^\otimes(\T)^\vee$, a basis of open sets is given by the complements of the compact open subsets of $\msf{Zg}^\otimes(\T)$, i.e., $(F)_\otimes^c$ is a basis of opens for $\msf{Zg}^\otimes(\T)^\vee$. 

Motivated by the dual topology on the Ziegler spectrum, also see \cref{GZperspective} for further discussion, we now define two Zariski-style topologies on $\closed{\T}$:
\begin{enumerate}
\item we define the \emph{dual topology}, $\closed{\T}^\vee$, to have a basis of open sets given by 
\[
(F)_\otimes^c=\{X \in \closed{\T} : F(C \otimes X) =0 \text{ for all $C \in \T^\c$}\}
\]
as $F$ ranges over $\msf{Coh}(\T)$;
\item we define the \emph{$\msf{GZ}$-topology}, $\closed{\T}^\msf{GZ}$, to have a basis of open sets given by \[[A]_\otimes = \{X \in \closed{\T} : \iHom(A,X) = 0\}\] as $A$ ranges over $\T^\c$.
\end{enumerate}
In both of these cases, we emphasize that we first take the closed points of the Kolmogorov quotient in the usual Ziegler $\otimes$-topology, and then we equip this subset with the alternative topologies. We now prove that these are actually well-defined topologies. The $\msf{GZ}$-topology will play a more crucial role in our study, but it is helpful to consider the dual topology too for comparative purposes.

\begin{lem}
Both the above topologies are well-defined on $\closed{\T}$.
\end{lem}
\begin{proof}
First observe that the $\msf{GZ}$-topology is a subtopology of the dual topology: when $F = \Hom_\T(A,-)$ we have $(F)_\otimes^c = [A]_\otimes$ as $\Hom_\T(A, C \otimes X) = 0$ for all $C \in \T^\c$ if and only if $\iHom(A,X) = 0$ as $\T$ is compactly generated. Therefore it suffices to prove the claim only for the dual topology. Accordingly, let $X\in\closed{\T}$ and suppose $Y$ is another point in $\msf{pinj}(\T)$ that is identified with $X$ in $\closed{\T}$. Suppose that $X\in (F)_{\otimes}^{c}$ and consider $\mc{F}=\{Y\in\T:F(C\otimes Y)=0\t{ for all }C\in\T^{\c}\}$, the corresponding $\otimes$-closed definable subcategory in $\T$. As $X\in\mc{F}$, it follows that $\msf{Def}^{\otimes}(X)\subseteq\mc{F}$. But, as $Y\sim X$, we also have $\msf{Def}^{\otimes}(Y)\subseteq\mc{F}$. Consequently $Y$ also lies in $(F)_{\otimes}^{c}$, as required. \end{proof}

\begin{rem}\label{GZperspective}
Let us motivate the above topologies, and show that they arise quite naturally. The letters $\msf{GZ}$ stand for Gabriel-Zariski. Recall from \cite[\S 13.1]{krspec} that for a locally coherent Grothendieck category $\A$, the Gabriel-Zariski (sometimes just Gabriel) spectrum has underlying set $\msf{inj}(\A)$, the indecomposable injective objects in $\A$, while a basis of open sets is given by
\[
[A]=\{X\in\msf{inj}(\A):\Hom_{\A}(A,X)=0\}
\]
where $A$ runs over $\msf{fp}(\A)$. It is known that, for such categories, the Hochster dual of $\msf{inj}(\A)$ equipped with the Ziegler topology is the Gabriel-Zariski spectrum, see \cite[Theorem 14.1.6]{psl}. 

As we have seen in \cref{purityflats}, the Ziegler spectrum of $\T$ is homeomorphic to the Ziegler spectrum of $\Flat{\T^{\c}}$, and this also restricts to a homeomorphism when we consider definable $\otimes$-closed subcategories. Consequently, should we wish to take the Hochster dual of the $\otimes$-closed Ziegler spectrum on $\T$, or equivalently $\Flat{\T^{\c}}$, we should consider the equivalent of the $\otimes$-closed Gabriel-Zariski spectrum of $\Flat{\T^{\c}}$. Note that $\Flat{\T^{\c}}$ is usually very far from abelian. However it is finitely accessible with a set of indecomposable injectives, so one can, using the usual definition, define the $\otimes$-closed Gabriel-Zariski spectrum on it; this is precisely the $\msf{GZ}$-topology defined above.

Of course $\Mod{\T^{\c}}$ is a locally coherent Grothendieck category. Were we to consider the $\otimes$-closed Gabriel-Zariski spectrum on $\Mod{\T^{\c}}$, we would obtain the topology $(-)^{\vee}$. It is just a striking property of $\Mod{\T^{\c}}$, proved in \cite[Theorem 1.9]{krsmash}, that the sets of indecomposable injective objects in $\Mod{\T^\c}$ and $\Flat{\T^{\c}}$ coincide. Consequently the Gabriel-Zariski spectrum of $\Mod{\T^\c}$ and $\Flat{\T^\c}$ have the same underlying sets, but $(-)^{\vee}$ is a finer topology than $\msf{GZ}$. 
\end{rem}

Homological support is tested against the pure injective objects $E_{\mc{B}}$. However, we show it is equivalent to test on any indecomposable pure injective $E \in \msf{Def}^\otimes(E_\mc{B})$. 

\begin{lem}\label{ihomequality}
Let $\mc{B} \in \hspec{\T^\c}$, and let $E$ be an indecomposable pure injective in $\msf{Def}^\otimes(E_\mc{B})$. Then for any $A \in \T^\c$, we have $\iHom(A, E_\mc{B}) = 0$ if and only if $\iHom(A, E) = 0$.
\end{lem}
\begin{proof}
The set $\msf{ker}(\iHom(A,-)) = \{X \in \T : \iHom(A,X) = 0\}$ is a $\otimes$-closed definable subcategory by \cref{definabletensorcharacterisation}. As $\msf{Def}^\otimes(E) = \msf{Def}^\otimes(E_\mc{B})$ by \cref{defequal}, the result follows.
\end{proof}

We may now relate these spaces to the homological spectrum. 
\begin{prop}\label{closedcts}
The map $\Phi\colon \hspec{\T^\c} \to \closed{\T}^\msf{GZ}$ is closed and continuous. 
\end{prop}
\begin{proof}
By \cref{ihomequality}, the preimage of $[A]_\otimes$ is 
$\{\mc{B} \in \hspec{\T^\c} : \iHom(A, E_\mc{B}) = 0\}$, which is equal to $\mrm{supp}^\msf{h}(A)^c$. This is an open set in the homological spectrum by definition and hence the map is continuous. 

We now prove that the map is closed. Let $C \in \T^\c$, and consider a basic closed set $\mrm{supp}^\msf{h}(C) = \{\mc{B} \in \hspec{\T^\c} : \iHom(C, E_\mc{B}) \neq 0\}$ of the homological spectrum. We have 
$\Phi(\mrm{supp}^\msf{h}(C)) = \{E_\mc{B} : \iHom(C,E_\mc{B}) \neq 0\}$
by \cref{ihomequality}. Since $\Phi$ is a bijection by \cref{hspecandKQ}, we see that $\Phi(\mrm{supp}^\msf{h}(C)) = [C]_\otimes^c$ which is closed in $\closed{\T}^\msf{GZ}$ by definition, and as such the map $\Phi\colon \hspec{\T^\c} \to \closed{\T}^\msf{GZ}$ is closed, which completes the proof.
\end{proof}

The following is an immediate consequence of \cref{closedcts} and \cref{hspecandKQ}.

\begin{thm}\label{hspecisgz}
The map $\Phi\colon \hspec{\T^\c} \to \closed{\T}^\msf{GZ}$ is a homeomorphism. \qed
\end{thm}

\begin{rem}
It is unclear if the map $\hspec{\T^\c} \to \closed{\T}^\vee$ is continuous. If it were, then its inverse would be closed, and as such the composite $\closed{\T}^\vee \to \hspec{\T^\c} \to \closed{\T}^\msf{GZ}$ (which is just the identity) would be closed. We refer the reader to \cref{seperation} for further discussion around this.
\end{rem}

\begin{figure}[htb]
\centering
\begin{tikzcd}
{\closed{\T}^\vee} \ar[rr, "\mrm{id}"', yshift=-1mm] & & {\closed{\T}^\msf{GZ}} \ar[ll, dashed, yshift=1mm, "\mrm{id}"' {yshift=3pt}] \ar[ll, dashed, yshift=1mm, "\bullet" description] \\
& {\hspec{\T^\c}} \ar[ul, dashed, "{\Phi}" {yshift=-1pt}] \ar[ul, dashed, "\bullet" description, pos=0.5] \ar[ur, "{\Phi}"' {yshift=-1pt}] \ar[ur, "\bullet" description, pos=0.5]  &
\end{tikzcd}
\caption{Solid maps are continuous, whereas dashed maps are only functions. A bullet $\bullet$ on the map indicates that the map is closed (and hence open since all the maps are bijections).}
\label{maps}
\end{figure}

\subsection{Separation properties}\label{seperation}
In this section we investigate the separation properties of the dual topology and the $\msf{GZ}$-topologies described above, and then relate these to the Balmer spectrum. Firstly, we consider the dual topology.

\begin{prop}\label{dualKQisT0}
The space $\closed{\T}^\vee$ is a $T_0$-space.
\end{prop}

\begin{proof}
Let $X \neq Y \in \closed{\T}$. By definition, it suffices to find an open set $\mathscr{O}$ of $\closed{\T}^\vee$ such that $X \in \mathscr{O}$ but $Y \not\in \mathscr{O}$. Since $X \neq Y \in \closed{\T}$, we have $\msf{Def}^\otimes(X) \neq \msf{Def}^\otimes(Y)$, and hence either $Y \not\in \msf{Def}^\otimes(X)$ or $X \not\in \msf{Def}^\otimes(Y)$. Without loss of generality suppose that $Y \not\in \msf{Def}^\otimes(X)$. Then by \cref{definabletensorcharacterisation} there exists a coherent functor $F$ such that $F(C \otimes X) = 0$ for all $C \in \T^\c$, but $F(C' \otimes Y) \neq 0$ for some $C' \in \T^\c$. Consider the open set $\mathscr{O} = (F)_\otimes^c$ for this distinguished coherent functor $F$. Then, by construction, $X \in \mathscr{O}$ and $Y \not\in \mathscr{O}$ as required.
\end{proof}

One may then wonder if the space $\closed{\T}^{\msf{GZ}}$ is also $T_{0}$, which, by \cref{hspecisgz}, would be equivalent to the homological spectrum $\hspec{\T^{\c}}$ also being $T_{0}$. We now show that, whenever this is the case, we obtain an interesting consequence, related to the Balmer spectrum. 
As the Balmer spectrum $\msf{Spc}(\T^\c)$ is the terminal support theory, there is a continuous map $\phi\colon \hspec{\T^\c} \to \msf{Spc}(\T^\c)$ given by $\phi(\mc{B}) = \bm{y}^{-1}(\mc{B})$, see~\cite[Remark 3.4]{Balmernilpotence}. 

The following two results give an alternative, and independent, proof of \cite[Theorem A]{BHScomparison}. We note that the following statement and proof do not require that compact and rigid objects coincide.

\begin{thm}\label{T0comparison}
The map $\phi\colon\hspec{\T^{\c}}\to\msf{Spc}(\T^{\c})$ is injective if and only if $\hspec{\T^\c}$ is a $T_{0}$-space.
\end{thm}

\begin{proof}
An equivalent definition of the homological support is $\mrm{supp}^\msf{h}(C) = \{\mc{B} \in \hspec{\T^\c} : \bm{y}C \not\in \mc{B}\}$, see~\cite[Proposition 4.4]{Balmerhomsupp}. Therefore we have
\begin{align}\label{closedT0}
\hspec{\T^\c} \t{ is }T_{0} \iff \forall\mc{B}\neq\mc{B}'\in\hspec{\T^{\c}} \,\,\exists \,C\in\T^{\c} \t{ such that }\bm{y}C\in\mc{B} \t{ and }\bm{y}C\not\in\mc{B'}.
\end{align}
As the comparison map $\phi\colon\hspec{\T^{\c}}\to\msf{Spc}(\T^{\c})$ is defined by $\phi(\mc{B})=\{C\in \T^{\c}:\bm{y}C\in\mc{B}\}$, if $\phi$ were not injective, then there would be distinct homological primes $\mc{B}$ and $\mc{B'}$ such that for every $C\in\T^{\c}$ we have $\bm{y}C\in\mc{B}$ if and only if $\bm{y}C\in\mc{B}'$. But this would contradict the assumption that $\hspec{\T^\c}$ is not $T_{0}$ by \cref{closedT0}. Hence, if $\hspec{\T^\c}$ is $T_0$, then $\phi$ is injective.  
On the other hand, if $\phi$ is injective, then two distinct homological primes $\mc{B}$ and $\mc{B}'$ contain different finitely presented projective functors, hence $\hspec{\T^\c}$ is $T_{0}$, again by \cref{closedT0}.
\end{proof}

\begin{cor}\label{BHSequivalences}
The following are equivalent:
\begin{enumerate}
\item The comparison map $\phi$ is a homeomorphism;
\item the comparison map $\phi$ is a bijection;
\item the homological spectrum $\hspec{\T^{\c}}$ is a $T_{0}$-space;
\item $\closed{\T}^\msf{GZ}$ is a $T_0$-space.
\end{enumerate}
\end{cor}

\begin{proof}
The equivalence of $(1)$ and $(2)$ holds as $\phi$ is seen to be closed and continuous by definition. The equivalence of $(2)$ and $(3)$ follows from \cref{T0comparison} together with the fact that $\phi$ is surjective by~\cite[Corollary 3.9]{Balmernilpotence}. Conditions (3) and (4) are equivalent since $\closed{\T}^\msf{GZ}$ and $\hspec{\T^\c}$ are homeomorphic by \cref{hspecisgz}. 
\end{proof}

One sees from the proof of \cref{dualKQisT0}, that $\closed{\T}^\vee$ is a $T_0$-space because definable subcategories are determined by coherent functors. In general however, definable subcategories cannot be distinguished only by finitely presented projective modules, so one might expect that $\closed{\T}^\msf{GZ}$ should not be $T_0$. 

However, the comparison map $\phi$ is a bijection in all known examples, see~\cite[\S 5]{Balmernilpotence}. Consequently, the \emph{simplicity} of the definable $\otimes$-closure of points in $\closed{\T}$ (as proved in \cref{simple}) must be the crucial feature which forces $\closed{\T}^\msf{GZ}$ to be $T_0$ in all the known examples. We hope that this observation sheds lights on the relation between the homological spectrum and the Balmer spectrum beyond the known examples in future.

\section{An example: the spectrum of $\D(\Z)$}\label{sec:example}
In this section, as an illustrative example we compute the homological spectrum of $\D(\Z)^\c$ using the method introduced above. It is shown in \cite[Theorem 8.1]{prestgarkusha} that
\[
\msf{Zg}(\D(\Z))\simeq \coprod_{n\in\Z}\msf{Zg}(\Z)
\]
where $\msf{Zg}(\Z)$ is the usual Ziegler spectrum of $\ab$. Consequently, an object in $\D(\Z)$ is indecomposable pure injective if and only if it is isomorphic to $X[n]$, where $X\in\msf{pinj}(\Z)$, $n\in\Z$, and $X[n]$ denotes the stalk in degree $n$.

The Ziegler spectrum of $\Z$ is well known~\cite[\S 5.2.1]{psl}. Its points are the following:

\begin{enumerate}\label{listzgz}
\item the quotient rings $\Z/p^{i}$ for each prime $p$ and $i > 0$;
\item the Pr\"{u}fer groups $\Z/p^{\infty}$ for each prime $p$;
\item the $p$-adic integers $\widehat{\Z_{p}}$ for each prime $p$;
\item the rationals $\Q$.
\end{enumerate}
We shall also need the following fact: the closed points of $\msf{Zg}(\Z)$ are the $\Z/p^{i}$ for each $1 \leq i < \infty$ and $p$ prime, as well as $\Q$. 

We will use this information to describe $\msf{Def}^{\otimes}(X)\subseteq\D(\Z)$, where $X\in\D(\Z)$ is an indecomposable pure injective. Note that since any $\otimes$-closed definable subcategory is closed under shifts, we see that $\msf{Def}^{\otimes}(X)=\msf{Def}^{\otimes}(X[n])$ for all $n\in\Z$, so we may assume without loss of generality that $X$ is a module concentrated in degree $0$.

Given a collection of objects $\msf{Y} \subseteq \ab$, we write $\msf{Def}_\Z(\msf{Y})$ for the smallest definable subcategory of $\ab$ containing $\msf{Y}$. We will use the fact that homology preserves definable building, see \cite{definablefunctors}. More explicitly, if $\msf{X}\subseteq\D(\Z)$ is a collection of objects and $U\in\msf{Def}(\msf{X})$, then $H_{0}(U)\in\msf{Def}_\Z(H_{0}(X):X\in\msf{X})$. We now describe the definable closures of the points of $\msf{Zg}^{\otimes}(\Z)$.

\begin{enumerate}\setlength{\parindent}{0cm}\setlength{\parskip}{0.1cm}
\item Consider $\Z/p^{i}$ viewed in degree zero. We claim that 
\begin{equation}\label{ex:finitelength}
\msf{Def}^{\otimes}(\Z/p^{i})\cap\msf{pinj}(\D(\Z))=\coprod_{k\in\Z}\{\Sigma^k\Z/p^{j}:j\leq i\}.
\end{equation}
Firstly, if $j\leq i$, then $\Z/p^{i}\otimes^{\msf{L}}\Z/p^{j}\simeq (\Z/p^{j}\xrightarrow{0}\Z/p^{j})$, which contains $\Z/p^{j}[0]$ as a pure subobject. This gives one inclusion in \cref{ex:finitelength}. To show that we obtain no other indecomposable pure injectives, we use that homology preserves definable building.

Since $\msf{Def}^{\otimes}(\Z/p^{i})=\msf{Def}(C\otimes^{\msf{L}}\Z/p^{i}:C\in\D(\Z)^\c)$, it follows that if $X[0]\in\msf{Def}^{\otimes}(\Z/p^{i})$, we have $X\in\msf{Def}_{\Z}(H_{0}(C\otimes^{\msf{L}}\Z/p^{i}):C\in\D(\Z)^\c)$. As $\Z$ is hereditary and has finite global dimension, we have that
\begin{align*}
\msf{Def}_{\Z}(H_{0}(C\otimes^{\msf{L}}\Z/p^{i}):C\in\D(\Z)^\c) &=\msf{Def}_{\Z}(M \otimes \Z/p^{i}:M\in\mod{\Z})
\\
&=\msf{Def}_{\Z}(\Z/p^{j}:j\leq i)
\end{align*}
where the last equality follows from the structure theorem of finitely generated abelian groups.  

As $\Z/p^j$ is a closed point in $\msf{Zg}(\Z)$ for every $j<\infty$, it follows that $\msf{Def}_\Z(\Z/p^j \mid j \leq i) \cap \msf{pinj}(\Z)$ is closed, and hence cannot contain any other indecomposable pure injectives. Thus, by the definable building argument the equality of \cref{ex:finitelength} holds.

\item 
We now do the case of the $p$-Pr\"{u}fer group $\Z/p^{\infty}$, and show that \[
\msf{Def}^{\otimes}(\Z/p^{\infty})\cap\msf{pinj}(\D(\Z))=\coprod_{k \in \Z}(\{\Sigma^k\Z/p^{i}:1\leq i\leq\infty\}\cup\{\Sigma^k\Q\}).
\]
For any $i$, we have $\Z/p^i \otimes^\msf{L} \Z/p^\infty \simeq \Z/p^i[1]$, so by shifting we obtain $\Z/p^{i}$. We can obtain $\Q$ without using the tensor product. Indeed, since $\Q \in \msf{Def}_\Z(\Z/p^\infty)$~\cite[\S 5.2.3]{psl}, and since products, pure subobjects and pure quotients are computed degreewise, it follows that $\Q \in \msf{Def}(\Z/p^\infty)$. This gives the reverse inclusion. To see that we cannot obtain any other indecomposable pure injectives, we use a similar argument to the previous case. 

Again, it suffices to consider $\msf{Def}_{\Z}(M \otimes \Z/p^{\infty}:M\in\mod{\Z})=\msf{Def}_{\Z}(\Z/p^{i}:1\leq i\leq \infty)$. The only point that could possibly be picked up is the $p$-adics $\widehat{\Z_{p}}$, but this cannot happen by the description of the Ziegler topology given at \cite[p. 225]{psl}.  

\item 
We now do the case for $\widehat{\Z_{p}}$, the $p$-adic integers, and claim that \[
\msf{Def}^{\otimes}(\widehat{\Z_{p}})\cap\msf{pinj}(\D(\Z))=\coprod_{k \in\Z}(\{\Sigma^k\Z/p^{i} : 1\leq i\leq \infty\}\cup\{\Sigma^k\Q\}\cup\{\Sigma^k\widehat{\Z_{p}}\}).
\] 
We have that
\[
\widehat{\Z_{p}}\otimes^\msf{L} \Z/p^{i} \simeq \Z/p^{i} \quad \mathrm{and} \quad \widehat{\Z_{p}} \otimes^\msf{L} \Z/p^\infty \simeq \Z/p^\infty,
\]
hence we pick up all finite length points related to $p$ as well as the $p$-Pr\"{u}fer group. To see we get $\Q$, note that, as shown at \cite[p. 225]{psl}, $\Q\in\msf{Def}_{\Z}(\widehat{\Z_{p}})$, and therefore the same holds in $\D(\Z)$, by considering the definable building process in degree zero. It is impossible to pick up any indecomposable pure injective related to a prime $q\neq p$ as there are no homomorphisms from $\Z/q$ to $\widehat{\Z_{p}}$ or $\Z/p^{i}$ for any $i$, and therefore the above description is complete.

\item
Finally we consider $\Q$. If $X[0]\in\msf{Def}^{\otimes}(\Q)$, then $X\in\msf{Def}_\Z(\Q\otimes_{\Z}M:M\in\mod{\Z})=\msf{Def}_\Z(\Q)$. But $\msf{Def}_\Z(\Q)\cap\msf{pinj}(\Z)=\Q$. In particular, we see that $\msf{Def}^{\otimes}(\Q)\cap\msf{pinj}(\D(\Z))=\coprod_{k \in \Z}\{\Sigma^k\Q\}$.
\end{enumerate}

From the previous discussion we now have the following result.
\begin{thm}\label{specZ}
The points and closure relations of $\msf{KZg}^\otimes(\msf{D}(\Z)))$, with the quotient topology, is
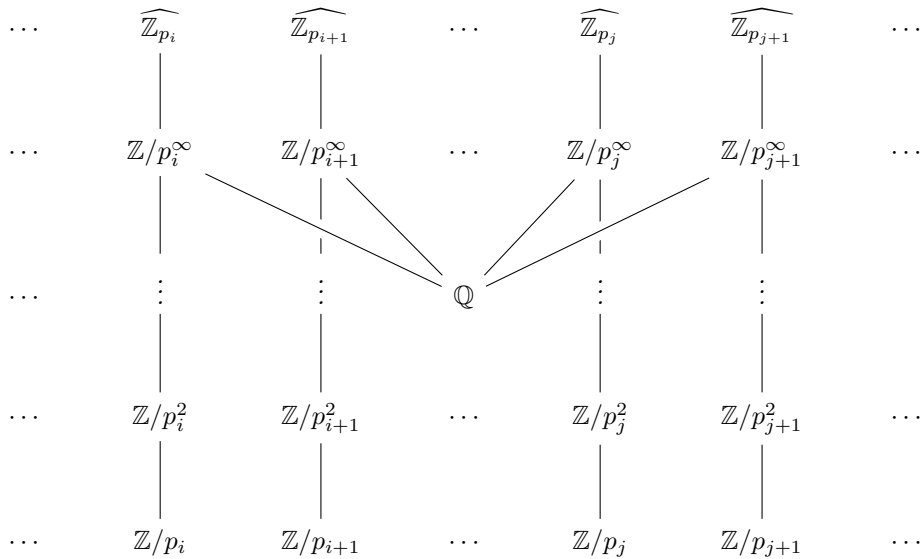
\[
\begin{tikzcd}[row sep=1cm]
\cdots &\widehat{\Z_{p_{i}}} \arrow[d, dash,]& \widehat{\Z_{p_{i+1}}} \arrow[d, dash]& \cdots& \widehat{\Z_{p_{j}}}\arrow[d, dash] & \widehat{\Z_{p_{j+1}}}\arrow[d, dash]& \cdots \\
\cdots &\Z/p_{i}^{\infty} \arrow[d, dash,] & \Z/p_{i+1}^{\infty}\arrow[d, dash]&\cdots & \Z/p_{j}^{\infty}\arrow[d, dash] & \Z/p_{j+1}^{\infty} \arrow[d, dash,]& \cdots \\
\cdots &\vdots& \vdots& \Q \arrow[ull, dash, crossing over] \arrow[ul, dash] \arrow[ur, dash] \arrow[urr, dash, crossing over] &\vdots &\vdots \\
\cdots &\Z/p_{i}^{2} \arrow[u,dash]& \Z/p_{i+1}^{2} \arrow[u,dash]&\cdots & \Z/p_{j}^{2} \arrow[u,dash]& \Z/p_{j+1}^{2}\arrow[u,dash]& \cdots\\
\cdots &\Z/p_{i} \arrow[u,dash]& \Z/p_{i+1} \arrow[u,dash]& \cdots& \Z/p_{j} \arrow[u,dash]& \Z/p_{j+1}\arrow[u,dash]& \cdots
\end{tikzcd}
\]
where closure is downwards.
\end{thm}

From the above result, we see that as a set we have
\[
\closed{\D(\Z)}=\{\Z/p : p \t{ prime}\}\cup\{\Q\}.
\]
Define $f\colon \closed{\D(\Z)}^\msf{GZ} \to \msf{Spec}(\Z)$ by $f(\Z/p) = (p)$ and $f(\Q) = (0)$. Recall that the only closed sets in $\msf{Spec}(\Z)$ are the whole space, and the finite unions of $(p_i)$ where each $p_i$ is prime. From this one sees that $f$ is continuous as $f^{-1}((p)) = [\Z/p]_\otimes^c$. In a similar way one verifies that $f$ is also closed, and hence $f\colon \closed{\D(\Z)}^\msf{GZ} \to \msf{Spec}(\Z)$ is a homeomorphism. 

Combining this with \cref{hspecisgz}, we see that $\hspec{\D(\Z)^\c}$ is homeomorphic to $\msf{Spec}(\Z)$ and hence also homeomorphic to the Balmer spectrum $\msf{Spc}(\D(\Z)^\c)$. As such, this example gives a very concrete proof that the homological and Balmer spectrum are homeomorphic for $\T = \D(\Z)$.

\bibliographystyle{abbrv}
\bibliography{references.bib}
\end{document}